\documentclass[12pt,reqno]{amsart}
\usepackage[numbers,sort&compress]{natbib}
\usepackage{amsfonts,bbm}
\usepackage{amssymb,color}
\usepackage{amssymb}
\usepackage{fancyhdr}
\usepackage[titletoc]{appendix}
\usepackage{enumitem}
\usepackage[colorlinks=true,backref]{hyperref}
\hypersetup{urlcolor=blue, citecolor=blue, linkcolor=black}

\usepackage[margin=3cm, a4paper]{geometry}

\newtheorem{thm}{Theorem}[section] 
\newtheorem{cor}[thm]{Corollary}
\newtheorem{lem}[thm]{Lemma}
 \newtheorem{prop}[thm]{Proposition}
\theoremstyle{remark} 
\newtheorem{rem}[thm]{Remark}

\newcommand{\EQ}[1]{\begin{align}\begin{split} #1 \end{split}\end{align}}
\newcommand{\EQn}[1]{\begin{align}\begin{split} #1 \end{split}\end{align}}
\newcommand{\EQnn}[1]{\begin{align} #1 \end{align}}

\newcommand{\CAS}[1]{\begin{cases} #1 \end{cases}}
\newcommand{\enu}[1]{\begin{enumerate}[label=(\alph*)] #1 \end{enumerate}}
\def\ncase#1{\left\{\begin{aligned}#1\end{aligned}\right.}

\newcommand{\Del}[1]{}

\def\norm#1{\left\|#1\right\|}
\def\normo#1{\left\|#1\right\|}

\def\abs#1{|#1|}
\def\aabs#1{\left|#1\right|}

\def\brk#1{\left(#1\right)}
\def\fbrk#1{\left\lbrace#1\right\rbrace} 
\def\jb#1{\langle#1\rangle}

\def\wb#1{\overline{#1}}

\def\pd{\partial}
\def\pdt{\partial_{t}}
\def\pdk{\partial_{k}}
\def\pdj{\partial_{j}}
\def\pdr{\partial_{r}}

\newcommand{\ra}{{\rightarrow}}

\def\loe{\leqslant}
\def\goe{\geqslant}
\def\lsm{\lesssim}
\def\gsm{\gtrsim}

\newcommand{\N}{{\mathbb N}}

\newcommand{\R}{{\mathbb R}}
\newcommand{\C}{{\mathbb C}}

\newcommand{\K}{{\mathcal{K}}}

\def\dx{\text{\ d} x}
\def\ds{\text{\ d} s}
\def\dt{\text{\ d} t}
\def\dive{\text{\rm div}}

\def\ep{\varepsilon}
\def\al{\alpha}
\def\be{\beta}

\def\om{\omega}
\def\ph{\varphi}

\def\de{\delta}
\def\De{\Delta}
\def\ka{\kappa}
\def\ta{\tau}
\def\la{\lambda}

\newcommand{\I}{\infty}
\def\rev#1{\frac{1}{#1}}
\def\half#1{\frac{#1}{2}}

\def\ut{u_{t}}

\def\pk{P_{k}}

\def\chir{\chi_{R}}

\def\jc{J^{(c)}}
\def\kab{K_{\al_1,\be_1}}

\def\es{E_{\text{S}}}
\def\ek{E_{\text{K}}}

\numberwithin{equation}{section} 

\title[2D NLS]{Scattering below the ground state for the 2D non-linear Schr\"{o}dinger and Klein-Gordon equations revisited}
\subjclass[2010]{35L70, 35B40, 35B44}
\keywords{Schr\"odinger equation, Klein-Gordon equation, Scattering}

\author{Zihua Guo}
\address{(Z. Guo) School of Mathematics, Monash University, Melbourne, VIC 3800, Australia}
\email{zihua.guo@monash.edu}

\author{Jia Shen}
\address{(J. Shen) School of Mathematical Sciences, Peking University, No 5. Yiheyuan Road, Beijing 100871, P.R.China}
\email{shenjia@pku.edu.cn}

\begin{document}
\maketitle

\begin{abstract}
We revisit the scattering problems for the 2D mass super-critical Schr\"{o}dinger and Klein-Gordon equations with radial data below the ground state in the energy space. We give an alternative proof of energy scattering for both defocusing and focusing cases using the ideas in [B. Dodson and J. Murphy, Proceedings of the American Mathematical Society, 145, 4859 (2017)]. Our results also include the exponential type nonlinearities which seems to be new for the focusing exponential NLS.
\end{abstract}


\section{Introduction}
In this note, we consider the non-linear Schr\"{o}dinger (NLS) equation
\EQn{\label{equ-main-nls}
i\pdt u -\De u = & f(u), \\ 
u(0,x) = & u_0(x),
}
and the non-linear Klein-Gordon (NLKG) equation
\EQn{\label{equ-main-kg}
	\pdt^2 u -\De u + u = & f(u), \\
	u(0,x) = & u_0(x), \\
	\ut(0,x) = & u_1(x),
}
where $u(t,x):\R\times\R^2\ra \C$. Throughout this paper we assume
\EQn{\label{assume-nl-term}
f(u)=\la \abs{u}^p u 	\quad \text{ or } \quad f(u) = \la \brk{e^{\ka_0\abs{u}^2}-1-\ka_0\abs{u}^2}u,
}
where $p>2$, $\kappa_0>0$, and $\la=1$ (focusing case) or $\la=-1$ (defocusing case). We define $F(u):\C\ra\R$ satisfies $F(0)=0$ and $\pd_{\wb{u}}F(u)=f(u)$, namely 
\EQn{\label{assume-potential-flow}
F(u)=\frac{2\lambda}{p+2}|u|^{p+2} \quad \mbox{ or } \quad F(u)=\frac{\lambda}{\ka_0} (e^{\ka_0\abs{u}^2}-1-\ka_0\abs{u}^2-\frac{\kappa_0^2}{2}|u|^4).} 
The NLS has conserved energy
\EQn{
	\es(u(t)) = \int_{\R^2} \half 1 \abs{\nabla u(t,x)}^2 - \half 1 F(u(t,x)) \dx,
}
and mass
\EQn{
	M(u(t)) = \int_{\R^2} \abs{u(t,x)}^2 \dx.
}
The NLKG has conserved energy
\EQn{
	\ek(u,\ut) = \int_{\R^2} \half 1 \abs{\nabla u(t,x)}^2 + \half 1 \abs{u(t,x)}^2 + \half 1 \abs{\ut(t,x)}^2 - \half 1 F(u(t,x)) \dx.
}

There are extensive results on the large data scattering in energy space for both NLS and NLKG. We only refer to the 2D results, in which case the classical (linear) Morawetz estimate breaks down even for the defocusing problems. We first recall the results for mass super-critical power type non-linear terms ($p>2$). For defocusing NLS and NLKG, Nakanishi \cite{nakanishi1999energy} introduced a new type of Morawetz estimate, and combined with induction on energy argument to prove scattering in  the energy space. For focusing NLKG, in \cite{ibrahim2011scattering}, Ibrahim, Masmoudi, and Nakanishi proved large data scattering for solutions with energy below ground state. Inui \cite{Inui} extended the results to the complex-valued Klein-Gordon equations.  Similar results for focusing NLS were also obtained in  \cite{fang2011scattering,AN10Kyoto} and \cite{guevara2013global}. The scattering of mass critical NLS ($p=2$) is more difficult and has also been solved (see \cite{KTV, Dodson1,Dodson2}), but this case is beyond our method in this paper.

Next, we recall the scattering results for the exponential type non-linear terms. This is the energy critical for 2D and is closely related to the Trudinger-Moser inequalilty.  Global well-posedness and scattering with small energy data was proved by Nakamura-Ozawa \cite{NaOz, NaOz2, NaOz3} (see \cite{wang2006energy} for more general nonlinearity). The exact size of data for well-posedness was investigated in \cite{CIMM} and a notion of criticality was proposed. In our notations, for NLS they proved global well-posedness for $E_S(u_0)\leq 2\pi/\ka_0$ and some ill-posedness for $E_S(u_0)>2\pi/\ka_0$. See \cite{IMM} and \cite{ibrahim2007ill} for the results on NLKG. For the large data scattering, it was proved in \cite{ibrahim2009scattering} for the defocusing NLKG with $\ek(u_0,0)\loe 2\pi/\ka_0$, and in \cite{ibrahim2012scattering} for the defocusing NLS  in the sub-critical region $\es(u_0) < 2\pi/\ka_0$. When $\es(u_0) = 2\pi/\ka_0$, the scattering was obtained in \cite{bahouri2014scattering} for radial data.  For focusing NLKG, scattering for solutions with energy below the ground state was proved in \cite{ibrahim2011scattering}.   It seems to us that the focusing exponential NLS was not studied.
Lastly, compared to NLS or NLKG results, the 2D defocusing exponential wave equation
\EQ{\label{eq:2Dwave}
\pd_t^2 u -\De u  + \brk{e^{|u|^2} - 1 - |u|^2}u =0
}
can be considered without any size or symmetry restriction. Sack and Struwe \cite{SS16MathAnn} established scattering for \eqref{eq:2Dwave} with arbitrary smooth and compactly supported initial data.

We remark that all the large data scattering results for the focusing problems mentioned above rely on Kenig-Merle's concentration compactness/rigidity method \cite{KM06Invent}. Recently, Dodson and Murphy (\cite{dodson2017new-radial, dodson2017new-nonradial}) used the ideas of combined virial and Morawetz estimates (first by Ogawa-Tsutsumi \cite{ogawa1991blow} for blowup problems) in the scattering problems. They gave a simple proof of the scattering for focusing $\dot{H}^{1/2}$-critical NLS in dimensions three and higher. For the two dimensional case, some new difficulty arises due to the weak time decay rate $t^{-1}$ of the linear propagator. Our purpose is to extend Dodson and Murphy's method to 2D. We exploit additional decay from the virial-Morawetz estimates to overcome the logarithmic divergence of time integral in 2D \footnote{We noticed that  similar results for 2D NLS with nonlinear term $|u|^pu$ was obtained very recently in \cite{Dodson2019} by similar ideas.}. 

In the focusing case $\la = 1$, we study the solutions with energy below the ground state. We introduce some notations on the variational analysis (from \cite{ibrahim2011scattering}) in the $d$ dimension although we will only need it for $d=2$. Let $(\al_1,\be_1)\in \R^2$ such that
\EQn{\label{range-al-be}
	\text{$\al_1\goe0$, $2\al_1+d\be_1\goe 0$, $2\al_1+(d-2)\be_1\goe 0$, and $(\al_1,\be_1)\ne(0,0)$.}
}
For $c\goe 0$ and $\ph\in H^1(\R^d)$, define the static energy
\EQn{
	\jc(\ph) := \half 1 \int_{\R^d} \abs{\nabla \ph}^2\dx + \half c \int_{\R^d} \abs{\ph}^2\dx - \half 1 \int_{\R^d} F(\ph)\dx.
}
Let 
\EQn{
	j_{\al_1,\be_1}^{(c)}(\la) = \jc(e^{\al_1\la}\ph(e^{-\be_1\la}x))
}
and
\EQn{
	\kab^{(c)}(\ph) = &  \pd_{\la}|_{\la=0}j_{\al_1,\be_1}^{(c)}(\la)\\
	= & \half{2\al_1 + (d-2)\be_1} \int \abs{\nabla \ph}^2\dx + \half{2\al_1 + \be_1 d} c \int \abs{ \ph}^2\dx\\
	& - \half{1} \int \brk{2\al_1 \Re\brk{\pd_u F(\ph) \ph} + d\be_1 F(\ph)} \dx.
}
We omit the super-script $c$ when $c=1$. 
We also take the quadratic part of $\kab$, i.e. 
\EQn{
	\kab^{\text{Q}}(\ph) = \half{2\al_1 + (d-2)\be_1} \int \abs{\nabla \ph}^2\dx + \half{2\al_1 + \be_1 d} \int \abs{ \ph}^2\dx.
}
Let
\EQn{
	m_{\al_1,\be_1} = \inf\fbrk{J(\ph):\text{$\ph\in H^1\backslash\fbrk{0}$ and $\kab(\ph)=0$}}.
}

The purpose of this paper is to give an alternative proof for the following theorems. The results were proved before without radial assumption (see \cite{ibrahim2009scattering, ibrahim2012scattering, ibrahim2011scattering}) except for the focusing exponential NLS. 
\begin{thm}
	Suppose that $u_0\in H^1(\R^2)$ is radial, $f(u)$ satisfies \eqref{assume-nl-term} with $p>2$ and $(\al_1,\be_1)\in\R^2$ satisfies \eqref{range-al-be}. Then 
	\enu{
	\item (Defocusing power type case)
	If $f(u)=-\abs{u}^p u$, then the solution of \eqref{equ-main-nls} exists globally and scatters.
	\item (Defocusing exponential case)
	If \EQ{f(u)=-\brk{e^{\ka_0\abs{u}^2}-1-\ka_0\abs{u}^2}u,} and $\es(u_0)<2\pi/\ka_0$, then the solution of \eqref{equ-main-nls} exists globally and scatters.
	\item (Focusing case)
	If $\la=1$, $\es(u_0) + M(u_0)/2<m_{\al_1,\be_1}$ and $\kab(u_0)\goe 0$, then the solution of \eqref{equ-main-nls} exists globally and scatters.
	} 
\end{thm}

\begin{thm}
	Suppose that $\brk{u_0,u_1}\in H^1(\R^2)\times L^2(\R^2)$ is radial, $f(u)$ satisfies \eqref{assume-nl-term} with $p>2$ and $(\al_1,\be_1)\in\R^2$ satisfies \eqref{range-al-be}. Then
	\enu{
		\item (Defocusing power type case)
		If $f(u)=-\abs{u}^p u$, then the solution of \eqref{equ-main-kg} exists globally and scatters.
		\item (Defocusing exponential case)
		If \EQ{f(u)=-\brk{e^{\ka_0\abs{u}^2}-1-\ka_0\abs{u}^2} u,} and $\ek(u_0,0)<2\pi/\ka_0$, then the solution of \eqref{equ-main-kg} exists globally and scatters.
		\item (Focusing case)
		If $\la=1$, $\ek(u_0,u_1)<m_{\al_1,\be_1}$ and $\kab(u_0)\goe 0$, then the solution of \eqref{equ-main-kg} exists globally and scatters. 	}
\end{thm}

\begin{rem}
\enu{
\item 

The above result for the focusing exponential NLS seems new. However, we comment that one may remove the radial assumption by using Kenig-Merle's approach \cite{KM06Invent} as in \cite{ibrahim2011scattering}. 

\item 
By variational analysis, under the assumption $\es(u_0) + M(u_0)/2<m_{\al_1,\be_1}$	or $\ek(u_0,u_1)<m_{\al_1,\be_1}$, $K_{\al_1,\be_1}(u_0)=0$ implies $u_0=0$. Note that $(0,u_1)$ can lead to non-trivial solution for NLKG. Therefore, to maintain consistency, we include the  $K_{\al_1,\be_1}(u_0)=0$ case in the above two theorems.
\item
In the focusing case, the threshold $m_{\alpha_1,\beta_1}$ is related to the ground state. For example, for power type nonlinearity, $m_{\alpha_1,\beta_1}=J(Q_0)$, 
where $Q_0$ is the unique radial solution of
\EQ{
	-\De Q_0 + Q_0 = \aabs{Q_0}^{p}Q_0.
}
For complex-valued Klein-Gordon equations with power type nonlinearity $|u|^pu$, we can obtain radial scattering below the standing wave solutions as Inui \cite{Inui}. More precisely, let $\om\in[0,1)$ and $Q_\om$ be the unique positive solution of 
\EQ{
	-\De Q_\om + \brk{1-\om^2}Q_\om=f(Q_\om).
}
Scattering for NLKG holds under the following assumptions: $p>2$, $(\al_1,\be_1)\in \eqref{range-al-be}$, $(u_0,u_1)\in H^1(\R^2)\times L^2(\R^2)$ is radial,
\EQ{
& E_K(u_0,u_1) - \omega \aabs{\Im\int_{\R^2} \wb{u}_0 u_1 \dx}  \\
< & \inf\fbrk{J^{(1-\om^2)}(\ph):\text{$\ph\in H^1\backslash\fbrk{0}$ and $\kab(\ph)=0$}}\\
= & J^{(1-\om^2)}(e^{\pm i\om t}Q_\om),
}
and $K_{\al_1,\be_1}(u_0)\goe 0$.

\item In view of the non-radial results in \cite{dodson2017new-nonradial}, it is natural and interesting to pursue whether one can remove the radial assumption in the 2D focusing case. However, our argument relies heavily on the radial symmetry to obtain the enhanced decay estimate from the virial-Morawetz estimate.  The reproof of non-radial NLKG problem is also unclear  in any dimension due to the lack of interaction Morawetz estimates.
}
\end{rem}

\section{Power type nonlinearity}
In this section, we take $f(u)=\lambda |u|^pu$ in equations \eqref{equ-main-nls} and \eqref{equ-main-kg} with $\la=\pm1$. 
To start with, we recall some classical Strichartz estimates for Schr\"odinger and Klein-Gordon equations (see \cite{KeelTao}).
\begin{prop}
Assume that $u_0(x)\in L^2(\R^d)$, then for any $(q,r)$ satisfying $2 \loe q,r\loe +\I$, $(q,r,d)\ne (2,\I,2)$ and
\EQ{
\rev{q}  \loe  \frac{d}{2}\brk{\half 1 - \rev{r}},
}
we have
\EQ{
\norm{e^{it\De}u_0}_{L_t^q B_{r,2}^{\frac{2}{q}+\frac{d}{r}-\frac{d}{2}}} \lsm& \norm{u_0}_{L_x^2} \label{esti-strz-nls}
}
and
\EQ{
\norm{e^{it\jb{D}}u_0}_{L_t^q B_{r,2}^{\beta(q,r)}} \lsm& \norm{u_0}_{L_x^2} \label{esti-strz-kg}
}
where
\EQ{
\beta(q,r)=\CAS{
-\frac{1}{q}+\frac{1}{r}-\frac{1}{2}, \quad 
\frac{d-1}{2}\brk{\half 1 - \rev{r}} \loe \rev{q}  \loe  \frac{d}{2}\brk{\half 1 - \rev{r}};\\
\frac{1}{q}+\frac{d}{r}-\frac{d}{2}, \quad 
\rev{q} \loe \frac{d-1}{2}\brk{\half 1 - \rev{r}}.
}
}
\end{prop}

In the focusing case, we will need some variational results. For the proof see for example \cite{ibrahim2011scattering}. 
\begin{thm}[\cite{ibrahim2011scattering}]
Assume $f(u)=|u|^pu$, $\frac{4}{d}<p<\frac{4}{d-2}$ ($\frac{4}{d}<p<\infty$ for $d=1,2$), and $(\al_1,\be_1)\in\R^2$ satisfies \eqref{range-al-be}. Then
\enu{
		\item
		\EQ{
		m_{\al_1,\be_1} = J(Q_0).
		}
		\item 
		We assume further $\al_1\ne 0$ for $d=2$. If $J(\ph) < m_{\al_1,\be_1}$ and $\kab(\ph)\goe 0$, we have 
		\EQ{
		\kab(\ph)\goe C\min\fbrk{m_{\al_1,\be_1}-J(\ph),\kab^{\text{Q}}(\ph)}.
		}
		\item 
		Define
		\EQ{
		\K_{\al_1,\be_1}^+:=\fbrk{\ph\in H^1: J(\ph)<m_{\al_1,\be_1}, K_{\al_1,\be_1}(\ph)\goe0},
		}
		and
		\EQ{
			\K_{\al_1,\be_1}^-:=\fbrk{\ph\in H^1: J(\ph)<m_{\al_1,\be_1}, K_{\al_1,\be_1}(\ph)<0}.
		}
		We have that $\K^\pm_{\al_1,\be_1}$ is independent of the choice of $(\al_1,\be_1)$, so we denote it by $\K^{\pm}$. Furthermore, $\K^+$ is a connected and open subset of $\fbrk{J(\ph)<m_{\al_1,\be_1}}$, and $0\in\K^+$.
		\item (Free energy equivalence)
		If $\ph\in \K^+$, we have that
		\EQ{
		J(\ph) \loe \half 1 \norm{\ph}_{H^1}^2 \loe \brk{1+\half d}J(\ph).
		}
	}
\end{thm}
Among all the functionals $\kab$, the virial functional $K_{d/2,-1}$ is useful for scattering. Define $G(u)=\Re\brk{\wb{u}f(u)-F(u)}$. By simple computation, 
\EQ{
K_{d/2,-1}(u) = \norm{\nabla u}_2^2 -\half d \int G(u) \dx = \norm{\nabla u}_2^2 - \la \frac{dp}{2(p+2)} \int |u|^{p+2} \dx.
}
We denote it by $K(u)$. 
Combining with the local well-posedness, this theorem yields global well-posedness and lower bound of virial functional for focusing problems. Now we state the variational result for power type focusing NLS and NLKG:
\begin{thm}[\cite{AN10Kyoto}]
	We assume that $\la=1$, and $(\al_1,\be_1)\in\R^2$ satisfies \eqref{range-al-be}. Let $4/d<p$ if $d=1$ or $d=2$, and $4/d<p<4/(d-2)$ if $d\goe 3$. If $\es(u_0) + M(u_0)/2<m_{\al_1,\be_1}$ and $\kab(u_0)\goe 0$, the solution of \eqref{equ-main-nls} exists globally, i.e. $u(t,x)\in C\brk{\R:H^1(\R^d)}$. Particularly, 
	\EQ{
	K(u(t)) \goe C\norm{\nabla u(t)}_2^2,
	}
	for all $t\in \R$.
\end{thm}

\begin{thm}[\cite{ibrahim2011scattering}]
	We assume that $\la=1$, and $(\al_1,\be_1)\in\R^2$ satisfies \eqref{range-al-be}. Let $4/d<p$ if $d=1$ or $d=2$, and $4/d<p<4/(d-2)$ if $d\goe 3$. If  $\ek(u_0,u_1)<m_{\al_1,\be_1}$ and $\kab(u_0)\goe 0$, the solution of \eqref{equ-main-kg} exists globally, i.e. $u(t,x)\in C\brk{\R:H^1(\R^d)}\cap C^1\brk{\R:L^2(\R^d)}$. Particularly, 
	\EQ{
		K(u(t)) \goe C\norm{\nabla u(t)}_2^2,
	}
	for all $t\in \R$.
\end{thm}
\subsection{Virial-Morawetz estimates}
In this section, we prove the virial-Morawetz inequality in both defocusing and focusing cases. 
Recall the definition \EQn{
G(u)=\Re\brk{\wb{u}f(u)-F(u)}.
}
We need the Morawetz identity for NLS and NLKG:
\begin{lem}[\cite{nakanishi2001remarks}]\label{lem-morawetz-identity}
	Suppose that $h(x):\R^d\ra\R^d$,  $q(x):\R^d\ra\R$, and $u(t,x):\R\times\R^d\ra \C$ are smooth. Let $h_j(x)$ be the $j$-th coordinate of the vector-valued function $h(x)$. We have that
	\EQ{
	& \Re\int\brk{i\pdt u -\De u - f(u)}\brk{h\cdot\nabla \wb{u} + q\wb{u}}\dx \\
	= & -\pdt\fbrk{\half 1 \Im \int u h\cdot\nabla \wb{u} \dx} + \sum_{k,j=1}^{d} \Re\int\pdk u \pdk h_j \pdj\wb{u}\dx -\half 1 \int\De q \abs{u}^2\dx \\
	&  -\half 1\int\dive h G(u)\dx + \Re\int\brk{q-\half 1\dive h}\brk{i\ut\wb{u}+\abs{\nabla u}^2-F(u)}\dx.
	}
	and
	\EQ{
	& \Re\int\brk{\pdt^2 u -\De u +u - f(u)}\brk{h\cdot\nabla \wb{u} + q\wb{u}}\dx \\
	= & -\pdt\fbrk{\half 1 \Im \int \ut \brk{h\cdot\nabla \wb{u} + q\wb{u}} \dx} + \sum_{k,j=1}^{d} \Re\int\pdk u \pdk h_j \pdj\wb{u}\dx -\half 1 \int\De q \abs{u}^2\dx \\
	&  -\frac{1}{2}\int\dive h G(u)\dx + \Re\int\brk{q-\half 1\dive h}\brk{ - \abs{\ut}^2 + \abs{u}^2 + \abs{\nabla u}^2-F(u)}\dx,
	}
\end{lem}
Next, we derive the key virial-Morawetz estimate following Dodson-Murphy \cite{dodson2017new-radial}. From now on, the constants may depend on the energy $\es$ or $\ek$ and mass $M(u)$, namely $C=C(\ek(u_0,u_1))$ or $C(\es(u_0),M(u_0))$.
\begin{prop}\label{prop-vi-mora-power}
	Assume $d\goe2$, $(\al_1,\be_1)\in\eqref{range-al-be}$, $f(u) = \la |u|^p u$ with $\la=\pm 1$,  $4/d<p<4/(d-2)$, $(u_0,u_1)\in H^1\times L^2$, radial. If $\la=1$, we assume further
	\EQ{
	\ncase{
	& \es(u_0) + M(u_0)/2 <m_{\al_1,\be_1},
	& \text{NLS case,} \\
	& \ek(u_0,u_1) <m_{\al_1,\be_1},
	& \text{NLKG case,} \\
}
}
	and $K(u_0)\goe 0$. Let  $u\in C\brk{I:H^1}$ be the global radial solution of \eqref{equ-main-nls} or \eqref{equ-main-kg}.
	
    Then, we have for any $R>0$ and $T_2>T_1>0$, 
	\EQn{
		\int_{T_1}^{T_2}\int \aabs{G(u)}\dx\dt \lsm 
		R + \brk{T_2-T_1}R^{-\min\fbrk{(d-1)p/2,2}}.
	}
	Furthermore, for any $\de>0$ and $T>0$, we have
	\EQn{\label{esti-vi-mora-power}
		\int_T^\I t^{-\max\fbrk{\frac{2}{2+(d-1)p},\rev3}-\de}\int \aabs{G(u)}\dx\dt \lsm T^{-\de},
	}
\end{prop}
\begin{proof}
We take $\chi(r)\in C_0^\I([0,+\I))$ such that $\chi(r)=1$ if $0\leq r \leq 1$ and $\chi(r)=0$ if $r\geq 2$. Let $\chir(r)=\chi(R^{-1}r)$ and
\EQ{
	\ph(r) = \int_0^r \chir^2(s)\ds.
}
Let $h(x)=\ph(\abs{x})x/\abs{x}$ and $q(x) = \dive h(x)/2$. After simple computation, we have
\EQ{
& \pdk h_j = \de_{kj} \frac{\ph(\abs{x})}{\abs{x}} - \frac{x_j x_k}{\abs{x}^2}\brk{\frac{\ph(\abs{x})}{\abs{x}} -\ph'(\abs{x})}, \\
& \dive h = \ph'(\abs{x}) + (d-1) \frac{\ph(\abs{x})}{\abs{x}},\\
& q(x) = \half d \ph'(\abs{x}) + \half{d-1}\brk{\frac{\ph(\abs{x})}{\abs{x}} -\ph'(\abs{x})}, \\
& \pdk \dive h = \frac{x_k}{\abs{x}}\ph''(\abs{x}) -  (d-1)\frac{x_k}{\abs{x}^2} \brk{\frac{\ph(\abs{x})}{\abs{x}} -\ph'(\abs{x})},
}
and
\EQ{
	\De \dive h = \ph'''(\abs{x}) + \frac{2(d-1)}{\abs{x}}\ph''(\abs{x}) - \frac{(d-1)(d-3)}{\abs{x}^2}\brk{\frac{\ph(\abs{x})}{\abs{x}} -\ph'(\abs{x})}.
}
Let $M(t)$ denote the Morawetz quantity:
\EQ{
	M(t) = 
	\ncase{
		& \half 1 \Im \int u h\cdot\nabla \wb{u} \dx,
		& \text{NLS case,} \\
		& \half 1 \Im \int \ut \brk{h\cdot\nabla \wb{u} + q\wb{u}} \dx,
		& \text{NLKG case.} \\
	}
}
It is easy to see that $M(t)\lsm R$.
Using Morawetz identity in Lemma \ref{lem-morawetz-identity}, for both NLS and NLKG, we have that
\EQ{
0 = & -\pdt M(t) + \sum_{k,j=1}^{d} \Re\int\brk{\de_{kj} \frac{\ph(\abs{x})}{\abs{x}} - \frac{x_j x_k}{\abs{x}^2}\brk{\frac{\ph(\abs{x})}{\abs{x}} -\ph'(\abs{x})}} \pdk u  \pdj\wb{u}\dx \\
& -\rev 4 \int\brk{\ph'''(\abs{x}) + \frac{2(d-1)}{\abs{x}}\ph''(\abs{x}) - \frac{(d-1)(d-3)}{\abs{x}^2}\brk{\frac{\ph(\abs{x})}{\abs{x}} -\ph'(\abs{x})}} \abs{u}^2\dx \\
	&  -\frac{1}{2}\int\brk{\ph'(\abs{x}) + (d-1) \frac{\ph(\abs{x})}{\abs{x}}} G(u)\dx ,
}
then
\EQ{
	\pdt M(t) = & \int \ph'(\abs{x}) \brk{\abs{\nabla u}^2-\frac{d}{2}G(u)}\dx \\
	& -\rev 4 \int \brk{\ph''' + \frac{2(d-1)}{\abs{x}}\ph''}\abs{u}^2\dx \\
	& + \int \brk{\frac{\ph(\abs{x})}{\abs{x}} -\ph'(\abs{x})} \brk{\abs{\nabla u}^2-\abs{\pdr u}^2}\dx \\
	& + \int \brk{\frac{\ph(\abs{x})}{\abs{x}} -\ph'(\abs{x})} \brk{\frac{(d-1)(d-3)}{4\abs{x}^2}\abs{u}^2- \half{d-1}G(u)}\dx.
}
Since $\ph/\abs{x}-\ph'=0$ and $\ph''=0$ if $\abs{x}\leq R$, there is no singularity at $x=0$ for terms containing $1/|x|$ in the above integrals.  
It is obvious that
\EQ{
\aabs{\int \brk{\ph''' + \frac{2(d-1)}{\abs{x}}\ph''}\abs{u}^2\dx} \lsm \rev{R^2}\norm{u}_2^2.
}
Since $u$ is radial, $\abs{\nabla u}^2-\abs{\pdr u}^2=0$. 

We are going to use radial Sobolev inequality to bound $G(u)$ when $\aabs{x}\goe R$. Note that
\EQ{
	\aabs{\int\brk{ G(u)-G(\chir u) }\dx} \lsm & \int \brk{\aabs{\pd_{u}G(u)} + \aabs{\pd_{\wb{u}}G(u)}}\aabs{u-\chir u}\dx \\
	\lsm & \int_{\abs{x}\goe R} \brk{\abs{\pd_{u}G(u)} + \abs{\pd_{\wb{u}}G(u)}}\abs{u} \dx.
}
Recall that $F(u)=2\la \abs{u}^p u/(p+2)$ and $G(u)=\la p\abs{u}^p u/(p+2)$. Therefore, by radial Sobolev inequality,
\EQ{
	\aabs{\int\brk{ G(u)-G(\chir u) }\dx} \lsm \int_{\abs{x}\goe R} \abs{u}^{p+2} \dx \lsm R^{-\half{d-1}p} \norm{u}_2^2. \label{esti-vimo-power-remainder-1}
}
We can also derive same bound for the following two terms
\EQn{
\int\brk{1-\chir^2}\aabs{G(u)}\dx \lsm \int_{|x|\goe R} \aabs{G(u)} \dx \lsm R^{-\half{d-1}p}, \label{esti-vimo-power-remainder-2}
}
and
\EQn{
\int\aabs{F(u)-F(\chir u)}\dx \lsm R^{-\half{d-1}p}. \label{esti-vimo-power-remainder-3}
}
Since $\ph/\abs{x}-\ph'=0$ if $\abs{x}\leq R$, and $0<\ph/\abs{x}-\ph'\lsm R/\abs{x}$ if $\abs{x}\geq R$, using radial Sobolev inequality again, the last term can be bounded by 
\EQ{
	& \int \brk{\frac{\ph(\abs{x})}{\abs{x}} -\ph'(\abs{x})} \brk{\frac{(d-1)(d-3)}{4\abs{x}^2}\abs{u}^2- \half{d-1}G(u)}\dx \\
	\lsm &  \rev{R^2} + \int_{|x|\goe R} \aabs{G(u)} \dx \lsm R^{-\min\fbrk{(d-1)p/2,2}}.
}

The main term comes from
\EQ{
	& \int \chir^2(\abs{x}) \brk{\abs{\nabla u}^2-\frac{d}{2}G(u)}\dx \\
	= & \int \brk{\abs{\nabla \brk{\chir u}}^2 -\half d G(\chir u) }\dx + \int \chir \De\brk{\chir} \abs{u}^2 \dx \\
	& + \half d \int\brk{1-\chir^2} G(u)\dx  - \half d \int\brk{ G(u)-G(\chir u) }\dx \\
	= & \int \brk{\abs{\nabla \brk{\chir u}}^2 -\half d G(\chir u) }\dx + \rev{R^2} O(\norm{u(t)}_2^2)  \\
	& + \half d \int\brk{1-\chir^2} G(u)\dx - \half d \int\brk{ G(u)-G(\chir u) }\dx.
}
When $\la=-1$, we clearly have 
\EQn{
	\int \brk{\abs{\nabla \brk{\chir u}}^2 -\half d G(\chir u) }\dx \gsm \int \aabs{G(\chir u) }\dx
}
When $\la=1$, it follows from the assumption and variational results that $u(t)$ is a global solution. We claim that there exists $R_0>0$ depending on the energy and mass of initial data, such that $\chir u(t)$ satisfies $\sup_t J(\chir u(t))<m_{\al_1,\be_1}$, for any $R>R_0$. In fact, by assumption, we have that
\EQ{
\sup_{t\in\R}J(u(t)) < m_{\al_1,\be_1}.
}
Since $\chir\loe 1$, we have
\EQ{
J(\chir u(t)) = & \half 1 \norm{\chir u(t)}_{H^1}^2 - \half 1 \int F(\chir u)\dx \\
= & \half 1 \norm{\chir \nabla u(t)}_{L^2}^2 + \half 1 \norm{\chir u(t)}_{L^2}^2 - \half 1 \int F(u)\dx \\ 
& - \half 1 \int \chir \De\brk{\chir} \abs{u}^2 \dx - \half 1 \int F(\chir u) -F(u) \dx \\
\loe & \half 1 \norm{\nabla u(t)}_{L^2}^2 + \half 1 \norm{ u(t)}_{L^2}^2 - \half 1 \int F(u)\dx\\ 
& + C\norm{u}_{L_t^\I L_x^2}^2 R^{-\min\fbrk{2,\half{d-1}p}} \\
\loe & J(u(t)) + CR^{-\min\fbrk{2,\half{d-1}p}},
}
where the constant $C$ depends only on the energy and mass. Therefore, we take $R_0$ such that
\EQ{
CR_0^{-\min\fbrk{2,\half{d-1}p}} < \half 1 \brk{m_{\al_1,\be_1}-J(u_0)}.
}
By the above claim, for any $t\in\R$, the set $\fbrk{\chir u(t):R_0<R}$ can be viewed as a continuous orbit in $\fbrk{J(\ph)<m_{\al_1,\be_1}}$. Since $u(t)\in\K^+$ is a limit point of the set $\fbrk{\chir u(t):R_0<R}$ and $\K^+$ is connected and open, we have that $K(\chir u(t)) \goe 0$ for all $t\in\R$ and $R>R_0$. By the choice of $R_0$, for any $R>R_0$, we have
\EQ{
m_{\al_1,\be_1} - \sup_{t\in\R} J(\chir u(t)) \goe \half 1 \brk{m_{\al_1,\be_1} - J(u_0)} = C,
}
where the constant is independent of $R$. Therefore, by variation estimate, for all $t\in\R$ and $R>R_0$,
\EQn{
	K(\chir u(t)) \goe & C\min\fbrk{m_{\al_1,\be_1}-\sup_{t\in\R}J(\chir u(t)), K^{Q}(\chir u(t))} \\
	\goe & C \min\fbrk{C, K^{Q}(\chir u(t))},
}
where the implicit constant is independent of $t$ and $R$. If $K^{Q}(\chir u(t))\goe C$, by Sobolev inequality
\EQ{
\int\aabs{G(\chir u(t))}\dx \lsm \int\aabs{G(u(t))}\dx \lsm \norm{u(t)}_{L_x^{p+2}}^{p+2} \lsm \norm{u(t)}_{L_t^\I H_x^1}^{p+2} \lsm 1,
}
we have $K(\chir u(t))\goe C \goe C \int G(\chir u(t)) \dx$. If $K^{Q}(\chir u(t))\loe C$, note that $K(\chir u(t))\goe0$ implies $K^Q(\chir u(t)) \goe \int G(\chir u(t)) \dx$, then
\EQ{
K(\chir u(t))\goe C K^{Q}(\chir u(t)) \goe C \int G(\chir u(t)) \dx.
}
Therefore, for both defocusing and focusing cases $\la=\pm1$, there exists $R_0>0$ such that for any $R>R_0$,
\EQ{
\int \chir^2(\abs{x}) \brk{\abs{\nabla u}^2-\frac{d}{2}G(u)}\dx \goe C \int G(\chir u(t)) \dx - CR^{-\min\fbrk{2,\half{d-1}p}},
}
where the constants depend only on the energy and mass.

Combine the above estimates, and integrate on the interval $[T_1,T_2]$, then we get
\EQ{
	\int_{T_1}^{T_2}\int \aabs{G(\chir u)}\dx\dt \lsm R + \brk{T_2-T_1}R^{-\min\fbrk{(d-1)p/2,2}},
}
for any $R>R_0$. In fact, the above inequality also holds for $R<R_0$, with the implicit constant depending on the conserved quantities. Since $R_0$ is a constant depending only on the energy and mass, for  $0<R<R_0$,
\EQ{
	\int_{T_1}^{T_2}\int \aabs{G(\chir u)}\dx\dt \lsm \brk{T_2-T_1} \lsm \brk{T_2-T_1}R^{-\min\fbrk{(d-1)p/2,2}}.
}
Moreover, using the estimate \eqref{esti-vimo-power-remainder-1},
\EQ{
	\aabs{\int_{T_1}^{T_2}\int \brk{G(u)-G(\chir u)}\dx\dt} \lsm  \brk{T_2-T_1}R^{-\min\fbrk{(d-1)p/2,2}},
}
then for any $0<T_1<T_2$ and $R>0$,
\EQ{
	\int_{T_1}^{T_2}\int \aabs{G(u)}\dx\dt \lsm R + \brk{T_2-T_1}R^{-\min\fbrk{(d-1)p/2,2}}.
}

Next, take $T_1=2^kT$, $T_2=2^{k+1}T$ for $k\in\N$ and
\EQ{
	R= \brk{2^k T}^{\max\fbrk{\frac{2}{2+(d-1)p},\rev3}},
}
then 
\EQ{
\int_{2^k T}^{2^{k+1}T}\int \aabs{G(u)}\dx\dt \lsm \brk{2^k T}^{\max\fbrk{\frac{2}{2+(d-1)p},\rev3}}.
}
Therefore, for any $\de>0$,
\EQ{
\int_{2^k T}^{2^{k+1}T} t^{-\max\fbrk{\frac{2}{2+(d-1)p},\rev3}-\de}\int \aabs{G(u)}\dx\dt \lsm \brk{2^k T}^{-\de}.
}
then the proposition follows by summing up the integral on intervals $[2^kT,2^{k+1}T]$, $k\in\N$.
\end{proof}

In the 2D case, there is a logarithmic problem due to the weak decay rate using the argument in \cite{dodson2017new-radial}.  We will exploit extra strength of the virial-Morawetz estimate.  
This is done by the following elementary lemmas. 

\begin{lem}\label{lem:eledecay}
Assume $f$ is non-negative and satisfies for some $\alpha\in (0,1)$
\EQ{
\int_1^\infty \frac{f(t)}{t^\alpha}dt<\infty.
}
Then for any $\epsilon, M>0$, there exists $T>M$ such that
\EQ{
\int_{T-\frac{1}{2\alpha}T^{1-\alpha}}^T f(t)dt<\epsilon.
}
\end{lem}
\begin{proof}
By change of variable $s=t^\alpha$, we have
\EQ{
\int_1^\infty \frac{f(s^{1/\alpha})}{\alpha s}s^{-1+1/\alpha}ds<\infty.
}
Then for any $\epsilon, M>0$, there exists $T_1>M$ such that 
\EQ{
\int_{T_1-1}^{T_1} f(s^{1/\alpha})s^{-1+1/\alpha}ds<\epsilon.
}
By change of variable back $s=t^\alpha$, we get 
\EQ{
\int_{(T_1-1)^{1/\alpha}}^{T_1^{1/\alpha}} f(t)dt<\epsilon.
}
Note that $T_1^{1/\alpha}-(T_1-1)^{1/\alpha}=T_1^{1/\alpha}[1-(1-\frac{1}{T_1})^{1/\alpha}]\geq \frac{1}{2\alpha}T_1^{1/\alpha-1}$. Taking $T=T_1^{1/\alpha}$, we complete the proof of the lemma.
\end{proof}

\begin{lem}\label{lem:elemint}
Assume $1\ll\ta\ll T$, $a\goe 1$ and $b>-1$. Then
	\EQ{
	\int_1^{T-\ta} \brk{T-s}^{-a} s^b \ds \lsm \ncase{
		& T^b\log\jb{T/\ta}, & \text{ when $a=1$,} \\
		& T^b \ta^{1-a}, & \text{ when $a>1$.}
		}
	}
\end{lem}
\begin{proof}
We have
\EQ{
	\int_1^{T-\ta} \brk{T-s}^{-a} s^b \ds\leq& (\int_1^{T/2}+\int_{T/2}^{T-\tau}) \brk{T-s}^{-a} s^b \ds\\
	\lsm &T^{-a}T^{b+1}+T^b \int_{T/2}^{T-\tau}|T-s|^{-a}ds\\
	\lsm &\ncase{
		& T^b\log\jb{T/\ta}, & \text{ when $a=1$,} \\
		& T^b \ta^{1-a}, & \text{ when $a>1$.}
		}
	}
\end{proof}

Let $\de=\de(p)>0$ be sufficiently small. We define the exponent $\al$ as
\EQn{
	\al=\max\fbrk{\frac{2}{2+(d-1)p},\rev3}+\de.
} 
By Lemma \ref{lem:eledecay}, we immediately get
\begin{cor}
	Let $d=2$. Suppose that all the assumptions in Proposition \ref{prop-vi-mora-power} hold. For any $\ep>0$, $T>0$, there exists $T_0=T_0(\ep,T)>T$,
	\EQn{\label{esti-decay-large-time-power}
		\int_{T_0-C T_0^{1-\al}}^{T_0} \int \aabs{G(u)}\dx\dt \lsm \ep.
	}
\end{cor}

\subsection{Proof of scattering}

We assume $u$ is a solution stated in Proposition \ref{prop-vi-mora-power}. Then $\norm{u}_{H^1}\lsm 1$. We will show some space-time bound.
We define
\EQ{
	\norm{u}_{S(I)} := 
	\ncase{
		& \norm{\jb{\nabla} u}_{L_{t,x}^{\frac{2(d+2)}{d}}(I\times \R^d)},
		& \text{NLS case,} \\
		& \norm{\jb{\nabla}^{1/2} u}_{L_{t,x}^{\frac{2(d+2)}{d}}(I\times \R^d)},
		& \text{NLKG case,} \\
	}
}
and
\EQ{
	\norm{u}_{W(I)} := \norm{u}_{L_{t,x}^{\frac{(d+2)p}{d}}(I\times \R^d)}.
}
We denote $W_T=W([T,\infty))$, similar for $S_T, L_{T,x}^\infty$.
The main task is to show: $\forall \epsilon>0$, $\exists T>0$ such that 
\EQ{\label{eq:weaknormdecay}
\norm{S(t-T)u(T)}_{W_T}<\epsilon, \quad \mbox{for NLS}\\
\norm{\dot S(t-T)u(T)+S(t-T)u_t(T)}_{W_T}<\epsilon, \quad \mbox{for NLKG}
}
where $S(t)=e^{it\Delta}$ (NLS) or $S(t)=\frac{\sin t\sqrt{-\Delta}}{\sqrt{-\Delta}}$ (NLKG). 
Once we have \eqref{eq:weaknormdecay}, then by the integral equation (e.g. for NLS)
\EQ
{
u(t)=S(t-T)u(T)-i\int_{T}^t S(t-s)f(u)ds
}
we get the following type of estimates: for some $a> 1$ and $b,a',b'>0$
\EQ{\label{eq:estWISI}
\norm{u}_{W_T}\lsm& \epsilon+\norm{u}_{W_T}^{a} \norm{u}_{S_T}^b,\\
\norm{u}_{S_T}\lsm& 1+\norm{u}_{W_T}^{a'} \norm{u}_{S_T}^{b'},\\
}
from which we get $\norm{u}_{S(\R)}<\infty$. In 2D case, we have  \eqref{eq:estWISI} hold with $a=a'=p>2$ and $b=b'=1$. Thus scattering follows. 

Now we prove \eqref{eq:weaknormdecay}. For NLS, we have
\EQ{
S(t-T)u(T)=&S(t)u_0-i\int_0^T S(t-s)f(u)ds\\
=&S(t)u_0-i\int_0^{T-\tau} S(t-s)f(u)ds-i\int_{T-\tau}^TS(t-s)f(u)ds\\
:=&I+II+III
}
where $\tau=\frac{1}{2\alpha}T^{1-\alpha}$. For term $I$, it is obvious that $\exists T>0$ such that $\norm{I}_{W_T}<\epsilon$. For terms $II$, by $\int_{T_1}^{T_2}S(t-s)f(u)ds=S(t-T_2)u(T_2)-S(t-T_1)u(T_1)$, we get
$\norm{II}_{S_T}\lsm 1$.
Thus by interpolation, it suffices to show 
\EQ{
\norm{II}_{L_{T,x}^\infty}< \epsilon.
}
We will use \eqref{esti-vi-mora-power}. 
It follows from H\"{o}lder and radial Sobolev inequality that 
\EQ{
\norm{II}_{L_{T,x}^\I}
	\lsm & \int_1^{T-\ta} \rev{\abs{T-s}}\norm{u}_{p+1}^{p+1}\ds+\tau^{-1} \\
	\lsm & \int_1^{T-\ta} \rev{\abs{T-s}}\norm{u}_{L_{\abs{x}\loe s^\de}^{p+1}}^{p+1}\ds + \int_1^{T-\ta} \rev{\abs{T-s}}\norm{u}_{L_{\abs{x}\goe s^\de}^{p+1}}^{p+1}\ds+\tau^{-1}  \\
	\lsm & \int_1^{T-\ta} \rev{\abs{T-s}} s^{\frac{2\de}{p+2}} \norm{u}_{L_{x}^{p+2}}^{p+1}\ds + \int_1^{T-\ta} \rev{\abs{T-s}} s^{-\half 1 \brk{p-1}\de}\norm{u}_{2}^2\ds+\tau^{-1}.
}
The second term is bounded by $T^{-(p-1)\de/4}$ by Lemma \ref{lem:elemint}. For the first term, in order to cover the logarithmic divergence of the time integral in $s$, we use virial-Morawetz estimate \eqref{esti-vi-mora-power} with $T=1$, then
\EQ{
    \int_1^{T-\ta} s^{-\al}\norm{u}_{p+2}^{p+2}\ds \lsm 1.
}
By H\"{o}lder inequality and the above estimate and Lemma \ref{lem:elemint}, we have
\EQ{
& \int_1^{T-\ta} \rev{\abs{T-s}} s^{\frac{2\de}{p+2}} \norm{u}_{L_{x}^{p+2}}^{p+1}\ds \\
\lsm & \brk{\int_1^{T-\ta} \brk{T-s}^{-(p+2)} s^{2\de+\al(p+1)}\ds}^{\rev{p+2}} \brk{\int_1^{T-\ta} s^{-\al}\norm{u}_{p+2}^{p+2}\ds}^{\frac{p+1}{p+2}} \\
\lsm & \brk{\int_1^{T-\ta} \brk{T-s}^{-(p+2)} s^{2\de+ \al(p+1)} \ds}^{\rev{p+2}} \\
\lsm & T^{\frac{p+1}{p+2}\brk{2\al-1+2\de/(p+1)}} \lsm T^{-\frac{p+1}{2(p+2)}\min\fbrk{\frac{p-2}{p+2},\rev3}} <\epsilon,
}
by taking $T>0$ sufficiently large since $p>2$.

For term $III$, we also have $\norm{III}_{S_T}\lsm 1$ and we will use \eqref{esti-decay-large-time-power}.  By interpolation, it suffices to show $\norm{III}_{L_{T,x}^4}<\epsilon$. Using Strichartz estimates and interpolation we get
\EQ{
\norm{III}_{L_{T,x}^4}\lsm& \norm{|u|^pu}_{L_{t\in [T-\tau, T],x}^{4/3}}\\
\lsm&  \norm{u}_{L_{t\in [T-\tau, T],x}^{4(p+1)/3}}^{p+1}\lsm \norm{u}_{L_{t\in [T-\tau, T],x}^{p+2}}^{\frac{3p+6}{4}}\norm{u}_{L_{t\in [T-\tau, T]}^{\infty} H^1}^{\frac{p-2}{4}}<\epsilon.
}

Similarly, for NLKG we have
\EQ{
&\dot S(t-T)u(T)+S(t-T)u_t(T)\\
=&\dot S(t)u_0+S(t)u_1+\int_0^T S(t-s)f(u)ds\\
=&(\dot S(t)u_0+S(t)u_1)+\int_0^{T-\tau} S(t-s)f(u)ds+\int_{T-\tau}^T S(t-s)f(u)ds\\
:=&I'+II'+III'.
}
Term $I'$ is trivial to estimate. For term $II'$, as term $II$ we have
\EQ{
\norm{\jb{D}^{-1}II'}_{L^\infty_{T,x}}\lsm  \int_0^{T-\ta} \rev{\abs{T-s}}\norm{u}_{p+1}^{p+1}\ds<\epsilon.
}
Term $III'$ is similar to term $III$. 

\section{Exponential type nonlinearity}
Throughout this section, we assume $d=2$ and  
\EQ{
f(u) = \la \brk{e^{\ka_0\abs{u}^2}-1-\ka_0\abs{u}^2}u
}
in equations \eqref{equ-main-nls} and \eqref{equ-main-kg} with $\la=\pm 1$. 
In order to deal with the exponential non-linear term, we need sharp Trudinger-Moser inequality:
\begin{lem}[\cite{ruf2005sharp}] We have
\EQ{ \label{esti-TM-ineq}
\sup_{u:\ \norm{\nabla u}_2^2+\norm{u}_2^2\leq 1}\int_{\R^2}(e^{4\pi |u|^2}-1)dx\leq C
}
\end{lem}

\begin{cor}\label{cor:Labound}
Let $a\geq 1$, $\ph\in H^1(\R^2)$ and $\norm{\nabla \ph}_2<\sqrt{4\pi a^{-1}}$. Then
	\EQn{
		\int_{\R^2}\brk{e^{\abs{\ph}^2}-1}^a\dx\lsm\frac{\norm{\ph}_2^2}{4\pi a^{-1}-\norm{\nabla\ph}_2^2}.
	}
\end{cor}
\begin{proof}
By \eqref{esti-TM-ineq} we have
\EQ{\label{eq:TMvar}
\sup_{u:\ \norm{\nabla u}_2^2+\norm{u}_2^2\leq 1}\int_{\R^2}(e^{4\pi a^{-1}|u|^2}-1)^adx\leq C
}
Fix $\ph\in H^1$ such that $\norm{\nabla \ph}_2<\sqrt{4\pi a^{-1}}$. Let $\ph_\lambda=\ph(\lambda x)$. Then
\EQ{
\norm{\nabla \ph_\lambda}_2^2+\norm{\ph_\lambda}_2^2=\norm{\nabla \ph}_2^2+\lambda^{-2}\norm{\ph}_2^2. 
}
Choose $\lambda>0$ such that $\norm{\nabla \ph_\lambda}_2^2+\norm{\ph_\lambda}_2^2=4\pi a^{-1}$. Then applying \eqref{eq:TMvar} with $u=\frac{\ph_\lambda}{\sqrt{4\pi a^{-1}}}$ we get
	\EQn{
		\int_{\R^2}\brk{e^{\abs{\ph(\lambda x)}^2}-1}^a\dx\leq C
	}
which implies 
	\EQn{
		\int_{\R^2}\brk{e^{\abs{\ph(x)}^2}-1}^a\dx\lsm \lambda^2=\frac{\norm{\ph}_2^2}{4\pi a^{-1}-\norm{\nabla\ph}_2^2}.
	}
\end{proof}
We also recall the radial Strichartz estimate for Klein-Gordon equation:
\begin{prop}[\cite{GHN18CMP}]
Suppose that $\ph\in L^2$ is radial, $d\goe2$, $2\loe q, r \loe +\infty$ and $\rev q < (d-\half 1)\brk{\half 1-\rev r}$.
	If $k\goe 0$, we have
	\EQ{\label{eq:radialstrichartz}
	\normo{e^{it\jb{D}}\pk\ph}_{L_t^q L_x^r} \lsm 2^{-\be(q,r)k} \normo{\pk\ph}_{2},
	}
	where
	\EQ{
    \beta(q,r)=\CAS{
    \frac{d}{2} -1 -\frac{1}{q}-\frac{d-2}{r}, \quad (d-1)\brk{\half 1-\rev r} < \rev q < (d-\half 1)\brk{\half 1-\rev r};\\
    \frac{1}{q}+\frac{d}{r}-\frac{d}{2}, \quad \rev q<(d-1)\brk{\half 1-\rev r},}
    }
	and if $k\loe 0$,
	\EQ{
	\normo{e^{it\jb{D}}\pk\ph}_{L_t^q L_x^r} \lsm 2^{\brk{\frac{2}{q}+\frac{d}{r}-\frac{d}{2}}k} \norm{\pk\ph}_{2}.
	}
Here $P_k$ is the Littlewood-Paley projector to the frequency of the size $\sim 2^k$.
\end{prop}

In the focusing case, we will need some variational results. For the proof see \cite{ibrahim2011scattering}. 
\begin{thm}[\cite{ibrahim2011scattering}]
	Suppose that  $(\al_1,\be_1)\in\R^2$ satisfies \eqref{range-al-be}. Then, 
	\enu{
		\item 
		\EQ{
		m_{\al_1,\be_1}  < 2\pi/\ka_0.
		}
		\item (Variational estimate)
		We assume futher $\al_1\ne 0$. If $J(\ph) < m_{\al_1,\be_1}$ and $\kab(\ph)\goe 0$, we have 
		\EQ{
		\kab(\ph)\goe\min\fbrk{C\brk{m_{\al_1,\be_1}-J(\ph)},C\kab^{\text{Q}}(\ph)}.
		}
		\item 
		Define
		\EQ{
		\K_{\al_1,\be_1}^+:=\fbrk{\ph\in H^1: J(\ph)<m_{\al_1,\be_1}, K_{\al_1,\be_1}(\ph)\goe0},
		}
		and
		\EQ{
			\K_{\al_1,\be_1}^-:=\fbrk{\ph\in H^1: J(\ph)<m_{\al_1,\be_1}, K_{\al_1,\be_1}(\ph)<0}.
		}
		We have that $\K^\pm_{\al_1,\be_1}$ is independent of the choice of $(\al_1,\be_1)$, so we denote it by $\K^{\pm}$. Furthermore, $\K^+$ is a connected and open subset of $\fbrk{J(\ph)<m_{\al_1,\be_1}}$, and $0\in\K^+$.
		\item (Free energy equivalence)
		If $\ph\in \K^+$, we have that
		\EQ{
		J(\ph) \loe \half 1 \norm{\ph}_{H^1}^2 \loe \brk{1+\half d}J(\ph).
		}
	}
\end{thm}
Similar as the power type case, define $G(u)=\Re\brk{\wb{u}f(u)-F(u)}$, and we have 
\EQ{
K_{d/2,-1}(u) = \norm{\nabla u}_2^2 -\half d \int G(u) \dx.
}
We denote it by $K(u)$. Recall from the definition \eqref{assume-nl-term} and \eqref{assume-potential-flow}, we have
\EQn{\label{eq:defn-G-exponential}
	G(u) = \la \rev{\ka_0} \brk{e^{\ka_0 |u|^2}\brk{ \ka_0\abs{u}^2-1}-\half 1 \ka_0^2\abs{u}^4}.
}

Combining with the local well-posedness, the above theorem yields global well-posedness and lower bound of virial functional for focusing problems. 
Now we state the variational result for 2D exponential focusing NLKG:
\begin{thm}[\cite{ibrahim2011scattering}]
	We assume that $\la=1$, and $(\al_1,\be_1)\in\R^2$ satisfies \eqref{range-al-be}. If  $\ek(u_0,u_1)<m_{\al_1,\be_1}$ and $\kab(u_0)\goe 0$, the solution of \eqref{equ-main-kg} exists globally, i.e. $u(t,x)\in C\brk{\R:H^1(\R^2)}\cap C^1\brk{\R:L^2(\R^2)}$. Particularly, we have $\sup_{t\in I} \norm{\nabla u}_2^2 < 4\pi/\ka_0$ and
	\EQ{
		K(u(t)) \goe C\norm{\nabla u(t)}_2^2,
	}
	for all $t\in \R$.
\end{thm}
The proof for NLS is almost the same as NLKG case, and we give the details for completeness.
\begin{thm}
	We assume that $\la=1$, and $(\al_1,\be_1)\in\R^2$ satisfies \eqref{range-al-be}. If $\es(u_0) + M(u_0)/2<m_{\al_1,\be_1}$ and $\kab(u_0)\goe 0$, the solution of \eqref{equ-main-nls} exists globally, i.e. $u(t,x)\in C\brk{\R:H^1(\R^d)}$. In particular, we have $\sup_{t\in \R} \norm{\nabla u}_2^2 < 4\pi/\ka_0$ and
	\EQ{
	K(u(t)) \goe C\norm{\nabla u(t)}_2^2,
	}
	for all $t\in \R$.
\end{thm}
\begin{proof}
    Suppose that $u(t)$ is the solution of \eqref{equ-main-nls} with maximal existence interval $I$. By conservation law, $J(u(t))=\es(u(t)) + M(u(t))/2<m_{\al_1,\be_1}$. If $K(u(t^*))=0$ for some $t^*\in I$, we have $u(t^*)=0\in\K^+$. Since $\K^+$ is an open set and $u(t)\in C\brk{I:H^1(\R^d)}$, $u(t)\in\K^+$ near $t^*$. Therefore, $u(t)\in\K^+$ for all $t\in I$. By free energy equivalence, $J(u(t))\sim \norm{u(t)}_{H^1}$. We then have $\sup_{t\in I} \norm{\nabla u}_2^2 < 4\pi/\ka_0$ using the functional $K_{0,1}$. Combining with local theory, we have that $I=\R$. Furthermore, $\norm{\nabla u}_2^2\loe C J(u(t)) = C J(u_0) \loe C\brk{ m_{\al_1,\be_1}-J(u_0)}$. By variation estimate, $K(u(t))\goe C \norm{\nabla u}_2^2$.
\end{proof}

\subsection{Virial-Morawetz estimates}
Next, we derive the key virial-Morawetz estimate. From now on, the constants may depend on $\ka_0$, the energy $\es$ or $\ek$ and mass $M(u)$, namely $C=C(\ka_0,\ek(u_0,u_1))$ or $C(\ka_0,\es(u_0),M(u_0))$.
\begin{prop}\label{prop-vi-mora-expo}
Assume $d=2$, $(\al_1,\be_1)\in\eqref{range-al-be}$, $(u_0,u_1)\in H^1\times L^2$, radial, and $f(u) = \la \brk{e^{\ka_0\abs{u}^2}-1-\ka_0\abs{u}^2}u$.
    If $\la = -1$, we assume  $\es(u_0)<2\pi/\ka_0$ for NLS and $\ek(u_0,u_1)<2\pi/\ka_0$ for NLKG; if $\la=1$, we assume 
	\EQ{
	\ncase{
	& \es(u_0) + M(u_0)/2 <m_{\al_1,\be_1},
	& \text{NLS case,} \\
	& \ek(u_0,u_1) <m_{\al_1,\be_1},
	& \text{NLKG case,} \\
}
}
and $K(u_0)\goe 0$. 
Let $u\in C\brk{\R:H^1}$ be the global radial solution of \eqref{equ-main-nls} or \eqref{equ-main-kg}.
	
    Then, for any $R>0$ and $T_2>T_1>0$,
	\EQn{
		\int_{T_1}^{T_2}\int \aabs{G(u)}\dx\dt \lsm 
		R + \brk{T_2-T_1}R^{-2}.
	}
	Furthermore, for any $\de>0$ and $T>0$, we have
	\EQn{\label{esti-vi-mora-expo}
		\int_T^\I t^{-\rev3-\de}\int \aabs{G(u)}\dx\dt \lsm T^{-\de},
	}
\end{prop}
\begin{proof}
We take $\chi(r)$,  $\chir(r)$, $\ph(r)$, $h(x)$, $q(x)$ and $M(t)$ as in power type case. 
It is easy to see that $M(t)\lsm R$.
Similar as before, by Morawetz identity in Lemma \ref{lem-morawetz-identity}, 
we have that
\EQ{
	\pdt M(t) = & \int \ph'(\abs{x}) \brk{\abs{\nabla u}^2-\frac{d}{2}G(u)}\dx \\
	& -\rev 4 \int \brk{\ph''' + \frac{2(d-1)}{\abs{x}}\ph''}\abs{u}^2\dx \\
	& + \int \brk{\frac{\ph(\abs{x})}{\abs{x}} -\ph'(\abs{x})} \brk{\abs{\nabla u}^2-\abs{\pdr u}^2}\dx \\
	& + \int \brk{\frac{\ph(\abs{x})}{\abs{x}} -\ph'(\abs{x})} \brk{\frac{(d-1)(d-3)}{4\abs{x}^2}\abs{u}^2- \half{d-1}G(u)}\dx,
}
and then 
\EQ{
\pdt M(t) = & \int \ph'(\abs{x}) \brk{\abs{\nabla u}^2-\frac{d}{2}G(u)}\dx \\
& - \half{d-1} \int \brk{\frac{\ph(\abs{x})}{\abs{x}} -\ph'(\abs{x})} G(u)\dx + O\brk{\rev{R^2}}.
}

When $|x|\goe R$, by radial Sobolev inequality $|u|\loe CR^{-1}$, we have $|e^{\ka_0|u|^2}-1-\ka_0|u|^2|\lsm \ka_0^2 |u|^4$. Then, we have
\EQn{\label{eq:pointwisebound-exponential-remainder}
\abs{\pd_u G(u)}+ \abs{\pd_{\wb{u}} G(u)} +\abs{F(u)} + \abs{G(u)}\lsm \rev{R^2} \abs{u}^2.
}
Then, we have
\EQnn{
	\aabs{\int\brk{ G(u)-G(\chir u) }\dx} \lsm R^{-2}, \label{esti-vimo-expo-remainder-1}\\
	\int\brk{1-\chir^2}\aabs{G(u)}\dx \lsm R^{-2}, \label{esti-vimo-expo-remainder-2}
}
and
\EQn{
\int\aabs{F(u)-F(\chir u)}\dx \lsm R^{-2}.\label{esti-vimo-expo-remainder3}
}
Therefore, the same as in power type case, by \eqref{esti-vimo-expo-remainder-1} and \eqref{esti-vimo-expo-remainder-2},
\EQ{
	\pdt M(t) = \int \brk{\abs{\nabla \brk{\chir u}}^2 -\half d G(\chir u) }\dx + O\brk{\rev{R^2}}.
}
When $\la=-1$, we have 
\EQn{
	\int \brk{\abs{\nabla \brk{\chir u}}^2 -\half d G(\chir u) }\dx \gsm \int \aabs{G(\chir u) }\dx
}
When $\la=1$, using the same argument in power case and \eqref{esti-vimo-expo-remainder3}, we obtain that there exists $R_0>0$ depending on the energy and mass of initial data, such that $\chir u(t)$ satisfies $\sup_t J(\chir u(t))<m_{\al_1,\be_1}$, for any $R>R_0$. Moreover, we have that $K(\chir u(t))\goe 0$ for all $t\in\R$ and $R>R_0$, and by variation estimate, 
\EQn{
	K(\chir u(t)) \gsm \int G(\chir u(t)) \dx,
}
where the implicit constant is independent of time $t$.

Combine the above estimates, and integrate on the interval $[T_1,T_2]$, then we get
\EQ{
	\int_{T_1}^{T_2}\int \aabs{G(\chir u)}\dx\dt \lsm R + \brk{T_2-T_1}R^{-2}.
}
In fact, the above inequality also holds for $R<R_0$, with the implicit constant depending on the conserved quantities. Indeed, it follows from the variational result that
\EQ{
    \sup_{t\in I} \norm{\nabla u}_2^2 < \frac{4\pi}{\ka_0},
} 
then using Trudinger-Moser inequality \eqref{esti-TM-ineq} with $\ph=\sqrt{\ka_0} u(t)$,
\EQ{
\int\aabs{\chir G(u(t))}\dx \loe 
C \int \brk{e^{\ka_0 \aabs{u(t)}^2}-1} \dx \loe C(\norm{u}_{L_t^\I H_x^1}).
}
Therefore, recall that $R_0$ is a constant depending only on the energy and mass, then for $0<R<R_0$,
\EQ{
	\int_{T_1}^{T_2}\int \aabs{G(\chir u)}\dx\dt \lsm \brk{T_2-T_1} \lsm \brk{T_2-T_1}R^{-2}.
}
Using the estimate \eqref{eq:pointwisebound-exponential-remainder} for $|x|\goe R$, we have for any $0<T_1<T_2$ and $R>0$,
\EQ{
	\int_{T_1}^{T_2}\int \aabs{G(u)}\dx\dt \lsm R + \brk{T_2-T_1} R^{-2}.
}
Take $T_1=2^kT$, $T_2=2^{k+1}T$ for $k\in\N$ and $R= \brk{2^k T}^{1/3}$, then the proposition follows by the summation procedure similar to the power type case.
\end{proof}
Let $\de>0$ be sufficiently small. We define the exponents $\al:=1/3 + \de$, and $\be:=1/2 + \de$. We need a global bound of $\norm{f(u)}_{L_x^1}$:
\begin{cor}
	Let $d=2$. Under the same assumptions in Proposition \ref{prop-vi-mora-expo}, for any $T>0$, we have global bound
	\EQn{\label{esti-global-bound-expo}
		\int_T^\I t^{-\be}\int\abs{f(u)}\dx\dt \lsm T^{-\de}.
	}
\end{cor}
\begin{proof}
From the definition \eqref{assume-nl-term} and \eqref{eq:defn-G-exponential}, when $|u|\gg 1$, we have $|f(u)|\lsm |u|e^{\ka_0 |u|^2}$, and $G(u)\gsm |u|^2e^{\ka_0|u|^2}$. Therefore, 
\EQn{
\lim_{u\ra0} \frac{|f(u)|}{|u|^5} =\half 1 \ka_0^2\text{ , and } \lim_{|u|\ra+\I} \frac{|f(u)|}{|G(u)|}=0.
}
Then, we have bound
\EQn{
|f(u)|\lsm |u|^5+|G(u)|.
}
By \eqref{esti-vi-mora-expo}, it remains to bound the $|u|^5$ term. By interpolation $\norm{u}_{L^5}\lsm \norm{u}_{L^6}^{9/10}\norm{u}_{L^2}^{1/10}$, \eqref{esti-vi-mora-expo} and H\"older inequality,
\EQ{
\int_{T}^\I t^{-\be}\norm{u(t)}_{L_x^5}^5 \dt \lsm & \int_{T}^\I t^{-\be}\norm{G(u(t))}_{L_x^1}^{\frac{3}{4}} \norm{u(t)}_{L_x^2}^{\half 1} \dt \\
\lsm  \brk{\int_T^\I t^{-\al}\norm{G(u(t))}_{L_x^1}  \dt}^{\frac34}&\brk{\int_T^\I t^{-1-\de}\norm{u(t)}_{L_x^2}^2\dt}^{\rev 4} \lsm T^{-\de}.
}
\end{proof}
\begin{cor}
	Let $d=2$. Suppose that all the assumptions in Proposition \ref{prop-vi-mora-expo} hold. Then for any $\ep>0$, $T>0$, there exists $T_0=T_0(\ep,T)>T$,
	\EQn{\label{esti-decay-large-time}
		\int_{T_0-10^{-1}T_0^{1-\al}}^{T_0} \int \aabs{G(u)}\dx\dt \lsm \ep.
	}
\end{cor}

\subsection{Proof of scattering}
The idea is similar to the power type case. For $q\in [1,\infty]$, $\epsilon\in \R$, we define $q(\epsilon)$ by
\EQ{
\frac{1}{q(\epsilon)}=\frac{1}{q}-\frac{\epsilon}{2}.
}
Then $(q(\epsilon))'=q'(-\epsilon)$.
We assume $u$ is a solution in Proposition \ref{prop-vi-mora-expo}. 
We have $\norm{u}_{H^1}\lsm 1$ and $\exists \eta>0$ such that
\EQ{
\sup_{t\in \R}\norm{\nabla u}_2^2<\frac{4\pi}{(1+100\eta)\kappa_0}.
}
By Corollary \ref{cor:Labound}, we get
\EQ{
\int_{\R^2} (e^{\kappa_0 |u|^2}-1)^{1(\theta\eta)}dx\lsm_\eta 1, \quad 0\leq \theta\leq 50. 
}
Using H\"older inequality, we get
\EQ{\norm{f(u)}_{L_x^1}\lsm 1.}
For any time interval $I\subset \R$, we define the strong Strichartz space for NLS,
\EQ{
	S(I) := L_t^\I H^1 \cap L_t^{2(\eta)} W_x^{1,\infty(-\eta)}(I\times \R^2),
}
and for NLKG,
\EQ{
\norm{u}_{S(I)} := \norm{u}_{L_t^\I H^1(I\times \R^2)} + \norm{P_{\goe 0}u}_{L_t^{2(\eta)} W_x^{1/2,\infty(-\eta/2)}(I\times \R^2)} + \norm{P_{\loe 0}u}_{L_t^{2} L_x^{\I}(I\times \R^2)}.
}
We also define the weak Strichartz space $W(I) := L_{t,x}^6(I\times \R^2)$. We first show: $\forall \epsilon>0$, $\exists T>0$ such that 
\EQ{\label{eq:weaknormdecayexpo}
\norm{S(t-T)u(T)}_{W_T}<\epsilon, \quad \mbox{for NLS}\\
\norm{\dot S(t-T)u(T)+S(t-T)u_t(T)}_{W_T}<\epsilon, \quad \mbox{for NLKG}
}

First consider NLS. As the power type case, there are three terms $I+II+III$. Similarly, $\norm{II}_{S_T}\lsm 1$. For the term $II$, noting that $\tau=cT^{1-\al}$
\EQ{
\norm{II}_{L_{T,x}^\I}
	\lsm & \int_1^{T-\ta} \rev{\abs{T-s}}s^\beta s^{-\beta}\norm{f(u)}_{1}\ds+\tau^{-1} \\
	\lsm & \tau^{-1}T^\beta+\tau^{-1}\leq \epsilon.}
Lastly we consider the term $III$.  We will use \eqref{esti-decay-large-time} which implies 
\EQ{\label{L6small}
\norm{u}_{L^6_{t,x}([T-\tau,T]\times \R^2)}<\epsilon.
}
Noting that
\EQ{
\norm{f(u)}_{L_t^{2(-\eta)}L_x^{1(\eta)}(I\times \R^2)}\lsm & \norm{\brk{e^{\ka_0\abs{u}^2}-1}\ka_0 |u|^2u}_{L_t^{2(-\eta)}L_x^{1(\eta)}(I\times \R^2)}\\
\lsm& \norm{\brk{e^{\ka_0\abs{u}^2}-1}}_{L_t^\I L_x^{1(49\eta)}}\norm{|u|^2u}_{L_t^{2(-\eta)}L_x^{\infty(-48\eta)}(I\times \R^2)}\\
\lsm& \norm{u}^3_{L_t^{6(-\eta/3)}L_x^{\infty(-16\eta)}(I\times \R^2)}\lsm \norm{u}_{L_{t,x}^{6}(I\times \R^2)}^{a_1}\norm{u}_{S(I)}^{3-a_1}
}
for some $a_1>1$
and
\EQ{
\norm{\nabla f(u)}_{L_t^{2(-\eta)}L_x^{1(\eta)}(I\times \R^2)}\lsm & \norm{\brk{e^{\ka_0\abs{u}^2}-1}\ka_0 \nabla(|u|^2u)}_{L_t^{2(-\eta)}L_x^{1(\eta)}(I\times \R^2)}\\
\lsm& \norm{\brk{e^{\ka_0\abs{u}^2}-1}}_{L_t^\I L_x^{1(49\eta)}}\norm{|u|^2\nabla u}_{L_t^{2(-\eta)}L_x^{\infty(-48\eta)}(I\times \R^2)}\\
\lsm& \norm{\nabla u}_{L_t^{2(\eta)}L_x^{\infty(-\eta)}(I\times \R^2)}\norm{u}^2_{L_t^{\infty(-\eta/2)}L_x^{\infty(-47\eta/2)}(I\times \R^2)}\\
\lsm&\norm{u}_{L_{t,x}^{6}(I\times \R^2)}^{a_2}\norm{u}_{S(I)}^{3-a_2}
}
for some $a_2>0$. 
Using the integral equation
\EQ
{
u(t)=S(t-T+\tau)u(T-\tau)-i\int_{T-\tau}^t S(t-s)f(u)ds
}
and the above estimates, we get
\EQ{
\norm{u}_{S([T-\tau,T])}\lsm 1+\sum_{j=1}^2\norm{u}_{L_{t,x}^{6}([T-\tau,T]\times \R^2)}^{a_j}\norm{u}_{S([T-\tau,T])}^{3-a_j}.
}
By \eqref{L6small} and continuity argument, we get
$\norm{u}_{S([T-\tau,T])}\lsm 1$
and thus
$\norm{III}_{W_T}<\epsilon$.
By the above estimate again and the integral equation, we can get the estimate of the type \eqref{eq:estWISI}. Thus scattering follows. 

For NLKG, the estimates are similar, but we use radial Strichartz estimate \eqref{eq:radialstrichartz} rather than $log$-Sobolev inequality to give a simpler proof. We remark that the radial symmetry is not necessary here. Term $II'$ is similar to term $II$. 
For term $III'$, using fractional chain rule, we have
\EQ{
	\norm{f(u)}_{L_t^{2} W_x^{1/2,1(\eta)}(I\times \R^2)}\lsm & \norm{\brk{e^{\ka_0\abs{u}^2}-1}\ka_0 |u|^2u}_{L_t^{2} W_x^{1/2,1(\eta)}(I\times \R^2)}\\
	\lsm& \norm{\brk{e^{\ka_0\abs{u}^2}-1}}_{L_t^\I L_x^{1(49\eta)}}\norm{|u|^2\jb{\nabla}^{1/2}u}_{L_t^{2} L_x^{\I(-48\eta)}(I\times \R^2)}\\
	\lsm& \norm{P_{\goe 0}\jb{\nabla}^{1/2}u}_{L_t^{2(\eta)} L_x^{\I(-\eta/2)}} \norm{|u|^2}_{L_t^{\I(-\eta)} L_x^{\I(-95\eta/2)}(I\times \R^2)}\\
	& + \norm{|u|^2P_{\loe 0}u}_{L_t^{2} L_x^{\I(-48\eta)}(I\times \R^2)} \\
	\lsm& \norm{u}_{S(I)} \norm{u}^2_{L_t^{\I(-\eta/2)} L_x^{\I(-95\eta/4)}}  + \norm{u}_{L_t^{6} L_x^{\I(-16\eta)}}^2
	\norm{P_{\leq 0}u}_{L_t^{6} L_x^{\I(-16\eta)}}\\
	\lsm & \norm{u}_{L_{t,x}^{6}(I\times \R^2)}^{b_1}\norm{u}_{S(I)}^{3-b_1} + \norm{u}_{L_{t,x}^{6}(I\times \R^2)}^{b_2}\norm{u}_{S(I)}^{3-b_2},
}
for some $b_1,b_2>0$. 

\section*{Acknowledgements}

Z. Guo was partially supported by ARC DP170101060. J. Shen was supported by China Scholarship Council 201706010021. The authors are very grateful to the anonymous referee for the careful reading and valuable suggestions which improve this paper.

\appendix
\section{Variation through Gagliardo-Nirenberg}
We provide an equivalent characterization of variation for power type non-linear terms when $\om=0$, which is based on Gagliardo-Nirenberg inequality.  The result is classical and equivalent to the sign functional approach, but we sketch the proof for completeness. Recall the sharp Gagliardo-Nirenberg(G-N) inequality from \cite{Nagy} and \cite{Weinstein}: let $0<p$ if $d=1$ or $d=2$, and $0<p<4/(d-2)$ if $d\goe 3$, for $g\in H^1$, we have
\EQ{
	\norm{g}_{p+2}^{p+2} \loe \frac{2(p+2)}{4-p(d-2)} \brk{\frac{pd}{4-p(d-2)}}^{-pd/4} \norm{Q_0}_2^{-p} \norm{g}_2^{p+2-pd/2} \norm{\nabla g}_2^{pd/4}, 
}
where $Q_0$ is the unique positive solution of
\EQ{
	-\De Q_0 + Q_0 = \aabs{Q_0}^{p}Q_0.
}
The equality holds if and only if $g(x)=Q_0(x)$ modulo some symmetries.
The ground state $Q_0$ has energy identity
\EQ{
	\norm{\nabla Q_0}_2^2 + \norm{Q_0}_2^2 = \norm{Q_0}_{p+2}^{p+2},
}
and Pohozaev identity
\EQ{
	\half{d-2}\norm{\nabla Q_0}_2^2 + \half d \norm{Q_0}_2^2 = \frac{d}{p+2}\norm{Q_0}_{p+2}^{p+2}.
}
Therefore, we obtain
\EQ{
	\norm{\nabla Q_0}_2^2 = \frac{pd}{4-p(d-2)} \norm{Q_0}_2^2,
}
and
\EQ{
	\norm{Q_0}_{p+2}^{p+2} = \frac{2(p+2)}{4-p(d-2)} \norm{Q_0}_2^2,
}
which gives another form of G-N inequality:
\EQ{
	\norm{g}_{p+2}^{p+2} \loe \frac{2(p+2)}{dp} \brk{\frac{\norm{g}_2^{p+2-pd/2} \norm{\nabla g}_2^{pd/2-2}}{\norm{Q_0}_2^{p+2-pd/2} \norm{\nabla Q_0}_2^{pd/2-2}}} \norm{\nabla g}_2^2.
}
Under the assumption $J(g)<J(Q_0)$, it is easy to prove that
\EQ{
	\norm{g}_2^{p+2-pd/2} \norm{\nabla g}_2^{pd/2-2}<\norm{Q_0}_2^{p+2-pd/2} \norm{\nabla Q_0}_2^{pd/2-2}
}
if and only if
\EQ{
	K(g) := K_{d/2,-1}(g) = \norm{\nabla g}_2^2 -\frac{dp}{2(p+2)} \norm{g}_{p+2}^{p+2}\goe 0.
}
Thus, we can give another form of main theorem for focusing power type NLS and NLKG:
\begin{thm}
	Let $d=2$, $f(u)=\abs{u}^p u$ with $p>2$. Suppose that $u_0\in H^1(\R^2)$ is radial and satisfies $\es(u_0) + M(u_0)/2<J(Q_0)$. 
	If $\norm{u_0}_2^{2} \norm{\nabla u_0}_2^{p-2}<\norm{Q_0}_2^{2} \norm{\nabla Q_0}_2^{p-2}$, the solution of \eqref{equ-main-nls} exists globally and scatters. 
\end{thm}
\begin{thm}
	Let $d=2$, $f(u)=\abs{u}^p u$ with $p>2$. Suppose that $(u_0,u_1)\in H^1\times L^2(\R^2)$ is radial and satisfies $\ek(u_0,u_1)<J(Q_0)$. If $\norm{u_0}_2^{2} \norm{\nabla u_0}_2^{p-2} < \norm{Q_0}_2^{2} \norm{\nabla Q_0}_2^{p-2}$, the solution of \eqref{equ-main-kg} exists globally and scatters. 
\end{thm}

\section*{Data Availability Statement}

Data sharing is not applicable to this article as no new data were created or analyzed in this
study.

\end{document}